\documentclass[12pt]{amsart}
\usepackage{amsmath,amsthm,amssymb,amsaddr,mathabx,subfigure,verbatim,fullpage,graphicx,charter,caption,color,mathtools}
\usepackage{hyperref}
\hypersetup{ colorlinks = true }
\newtheorem{theorem}{Theorem}

\newtheorem{lemma}{Lemma}
\newtheorem{claim}{Claim}
\newtheorem{definition}{Definition}

\newcommand\blfootnote[1]{%
\begingroup
\renewcommand\thefootnote{}\footnote{#1}%
\addtocounter{footnote}{-1}%
\endgroup}

\author{Jean Cardinal}
\address{Universit\'e libre de Bruxelles (ULB), Brussels, Belgium.\\
  {\tt jcardin@ulb.ac.be}}
\author{Stefan Langerman}
\address{Universit\'e libre de Bruxelles (ULB), Brussels, Belgium.\\
  {\tt slanger@ulb.ac.be}}
\author{Pablo P\'erez-Lantero}
\address{Universidad de Santiago, Santiago, Chile.\\
  {\tt pablo.perez.l@usach.cl}.}

\title{On the diameter of tree associahedra}
\begin{document}
\maketitle

\blfootnote{\begin{minipage}[l]{0.3\textwidth} \includegraphics[trim=10cm 6cm 10cm 5cm,clip,scale=0.15]{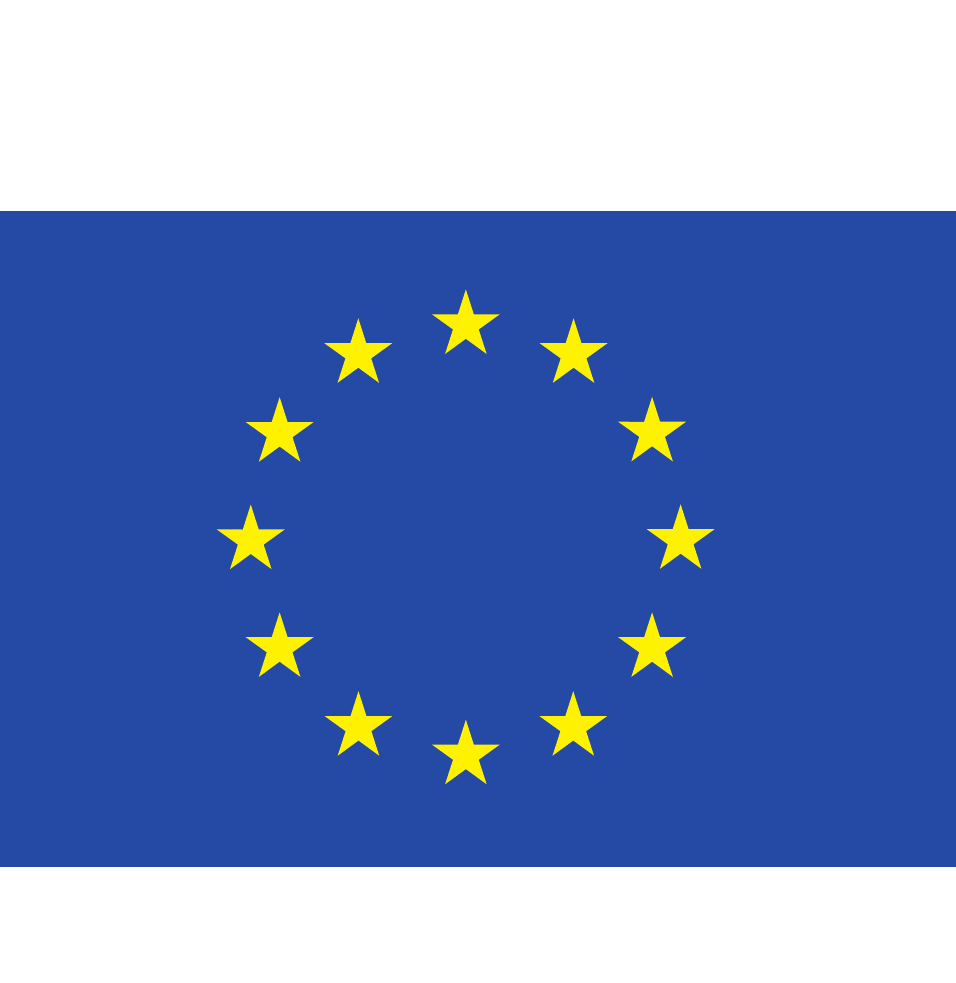} \end{minipage}  \hspace{-2cm} \begin{minipage}[l][1cm]{0.7\textwidth}This work has received funding from the European Union's Horizon 2020 research and innovation programme under the Marie Sk\l{}odowska-Curie grant agreement No 734922.
\end{minipage}}
\vspace{-0.9cm}

\begin{abstract}
We consider a natural notion of search trees on graphs, which we show is ubiquitous in various areas of discrete mathematics and computer science.
Search trees on graphs can be modified by local operations called rotations, which generalize rotations in binary search trees.
The rotation graph of search trees on a graph $G$ is the skeleton of a polytope called the {\em graph associahedron} of $G$.

We consider the case where the graph $G$ is a tree.
We construct a family of trees $G$ on $n$ vertices and pairs of search trees on $G$ such that the minimum number of rotations required to transform one search tree into the other is $\Omega (n\log n)$. This implies that the worst-case diameter of tree associahedra is $\Theta (n\log n)$, which
answers a question from Thibault Manneville and Vincent Pilaud.
The proof relies on a notion of projection of a search tree which may be of independent interest. 
\end{abstract}

\section{Introduction}

Rotations in binary trees are simple local operations that exchange the levels of two nodes of the tree, and allow the transformation of any tree into any other tree. Using the classical bijection between binary trees and triangulations of a convex polygon, we can map rotations to {\em flips} in such triangulations, which consist of replacing an edge shared by two triangles by the other diagonal of the quadrilateral that they form. 
These operations give rise to {\em rotation graphs} on binary trees, or {\em flip graphs} on triangulations. In those graphs, the set of vertices is the set of trees (respectively, triangulations), and two of them are adjacent if and only if they differ by a single rotation (flip). These graphs are known to be skeletons of {\em associahedra}, or {\em Stasheff polytopes}, which appear in many contexts in discrete geometry and algebraic topology~\cite{T51,S63,L89,CSZ15}. An illustration is given in Figure~\ref{fig:graph}.

A well-known challenging question, to which a final answer has only been given recently, is that of the {\em diameter} of associahedra: the largest number of rotations needed to transform a binary tree into another~\cite{STT86,D10,P14}. In 1986, Sleator, Tarjan, and Thurston~\cite{STT86} established a tight bound of $2n-4$ on the diameter of the $n$-dimensional associahedron, for large enough values of $n$, using notions from hyperbolic geometry. In 2014, Pournin~\cite{P14} completely settled the question of the exact value of the diameter with a purely combinatorial method, proving that it was equal to $2n-4$ for all $n\geq 9$.

\begin{figure}
\includegraphics[page=11,scale=.7]{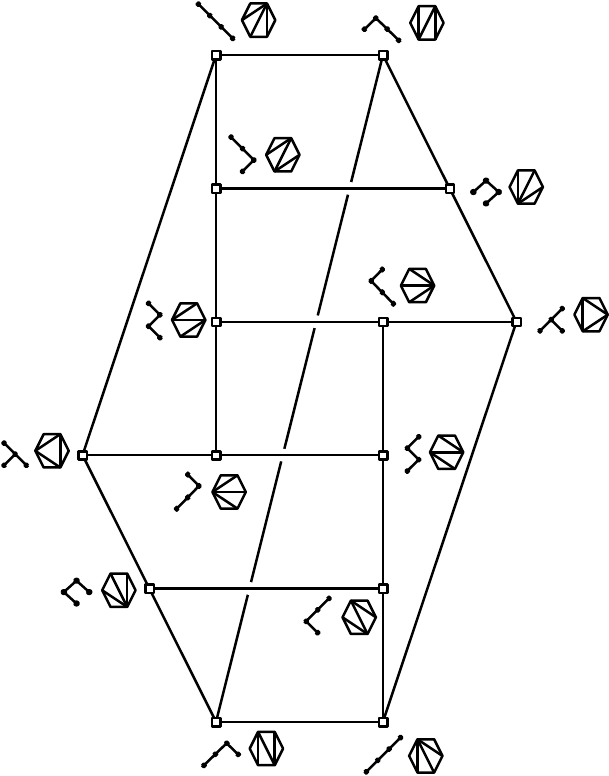}
\caption{\label{fig:graph}Rotations in binary trees and flips in triangulations.}
\end{figure}

Carr and Devadoss~\cite{CD06,D09} introduced {\em graph associahedra}, polytopes associated with graphs, that generalize associahedra, with a rich combinatorics related to Coxeter complexes and moduli spaces of curves. The usual associahedra are graph associahedra of paths, permutohedra are graph associahedra of complete graphs, and cyclohedra are graph associahedra of cycles.
The diameter of cylohedra was recently studied by Pournin~\cite{P17}.
Graph properties of the skeletons of graph associahedra were investigated by Manneville and Pilaud~\cite{MP15}. They proved in particular that the diameter of the graph associahedron of a graph $G$ is at least the number of edges of $G$, and at most quadratic in the number of vertices. They asked the following question: Does there exist a family of trees on $n$ vertices such that the diameter of their associahedra is $\Omega (n\log n)$?

Our contribution is an affirmative answer to this question. 

In Section~\ref{sec:def}, we give the definition of a {\em search tree} on a given graph $G$, which we show to be essentially equivalent to several other structures studied in various contexts. We also define the rotation operations in those trees, yielding a rotation graph that is the skeleton of a graph associahedron. We then restrict our attention to the case where $G$ is a tree and state the main result in Section~\ref{sec:statement}. The proof of our lower bound on the diameter of tree associahedra is given in Section~\ref{sec:proof}.

\section{Definitions}
\label{sec:def}

\subsection{Search trees}

\begin{figure}
\subfigure[A graph $G$.]{\includegraphics[page=1,scale=.5]{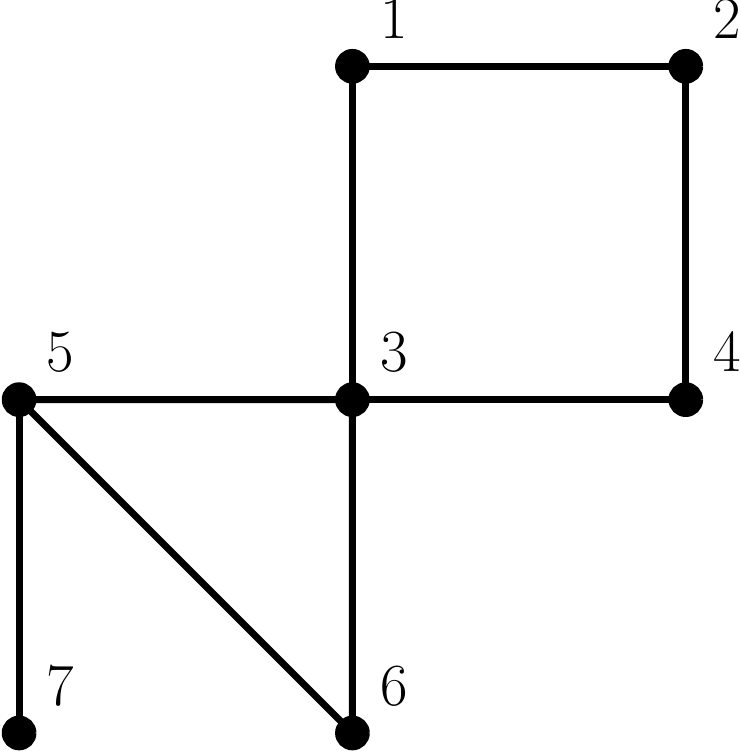}}
\hspace{1cm}
\subfigure[A search tree on $G$.]{\includegraphics[page=2,scale=.5]{figures.pdf}}
\hspace{1cm}
\subfigure[\label{fig:tubing}The corresponding tubing on $G$.]{\includegraphics[page=6,scale=.5]{figures.pdf}}
\caption{\label{fig:tree}An example of search tree.}
\end{figure}

In what follows, all graphs are simple and undirected. We consider {\em search trees} on graphs, defined recursively as follows.

\begin{definition}
Let $G=(V,E)$ be a connected graph. A {\em search tree} on $G$ is a rooted tree $T$ with vertex set $V$ and root $r\in V$.
If $|V|=1$, then $T$ is the single vertex $r$.
Otherwise, $G-r$ is composed of some $k\geq 1$ connected components.
Then $T$ is composed of the root $r$, connected to the roots of $k$ search trees, one for each of the components.
\end{definition}

An illustration of this definition is given in Figure~\ref{fig:tree}.
We also make use of the standard terminology on rooted trees. For each vertex $v\in V$, the first neighbor of $v$ on the path from $v$ to $r$ in $T$ is called the {\em parent} of $v$. The other neighbors of $v$ are called the {\em children} of $v$. The {\em subtree} rooted at $v$ is the tree formed by $v$ together with the subtrees of its children.

If $G$ is a path on $n$ vertices, then the search trees on $G$ are exactly the binary search trees on $n$ elements. If $G$ is a complete graph on $n$ vertices, then search trees on $G$ are in one-to-one correspondence with permutations of the vertices of $V$.

\subsection{Related structures}

Search trees on graphs as defined above are found in many disguises in the literature. 

In polyhedral combinatorics, search trees are essentially equivalent to inclusionwise maximal {\em tubings}, as introduced by Carr and Devadoss~\cite{CD06}. In short, tubings on a graph $G$ are collections of connected subgraphs of $G$, called {\em tubes}, such that every pair of tubes is either (i) nested or (ii) nonadjacent, that is, disjoint and such that their union is not connected. 
We obtain a one-to-one correspondence between search trees and inclusionwise maximal tubings by letting the tubes be induced by the set of vertices contained in the subtrees of a search tree. An example is given in Figure~\ref{fig:tubing}. Search trees are referred to as {\em spines} by Manneville and Pilaud~\cite{MP15}. Tubings on $G$ are faces of the graph associahedron of $G$. The vertices of the graph associahedron of $G$ are inclusionwise maximal tubings. The edges of the graph associahedron connect pairs of maximal tubings that differ by exactly two tubes, or equivalently pairs of search trees that differ by a single rotation, as we describe below. An independent description of graph associahedra based on a related notion of {\em graphical building sets} was given by Postnikov~\cite{P09}. Geometric realizations of graph associahedra have been described by Devadoss~\cite{D09}.

In combinatorial optimization, search trees can be identified under the terminology of {\em vertex rankings}~\cite{BDJKKMT98}. 
Vertex rankings are colorings $c:V\to [k]$ such that for any pair of vertices $u,v\in V$ such that $c(u)=c(v)$, and for every path between $u$ and $v$, there exists a vertex $w$ on this path such that $c(w)>c(u)$.
The {\em vertex ranking chromatic number} of $G$ is the smallest $k$ such that $G$ admits a $k$-ranking. 
We observe that every $k$-ranking of a connected graph $G$ directly yields a search tree $T$ of height $k$ on $G$. 
Indeed, since $G$ is connected, there cannot be more than one vertex $r$ of color $k$. Pick $r$ as the root of $T$, and recurse on the connected components of $G-r$. Conversely, every search tree can be interpreted as a vertex ranking. Computing optimal vertex rankings of arbitrary graphs is NP-hard~\cite{P88}. However, they can be computed in linear time for trees~\cite{S89}, and in polynomial time for a number of other graph classes~\cite{D93,A94}. For arbitrary graphs, only an  polynomial-time $O(\log^2 n)$-approximation algorithm is known~\cite{BGHK95}.
An elegant connection between the performance ratio of online hitting set algorithms and vertex rankings has been established by Even and Smorodinsky~\cite{ES14}.

In graph theory, the vertex ranking chromatic number is also known as the {\em tree-depth}. Connections with other classical structural parameters of graphs such as treewidth, pathwidth, and degeneracy, as well as the behavior of tree-depth with respect to minors and induced subgraphs have been studied extensively. We refer the reader to Ne\v{s}et\v{r}il and Ossona de Mendez~\cite{NO12} for a comprehensive survey. 


In algorithms, search trees have been studied from the data structure point of view. 
They provide a model of decision tree for search problems in which we wish to identify a hidden target vertex in a graph~\cite{LTT89,OP06,EKS16}. 
In this model, an oracle answers vertex queries: given a vertex $v$, it indicates the connected component of $G-v$ containing the target vertex. The height of the search tree is the worst-case query complexity of the search. A linear-time algorithm for constructing worst-case optimal search trees on trees has been rediscovered by Mozes, Onak, and Weimann~\cite{MOW08}.

\subsection{Rotations}

\begin{figure}
\includegraphics[page=3,scale=.5]{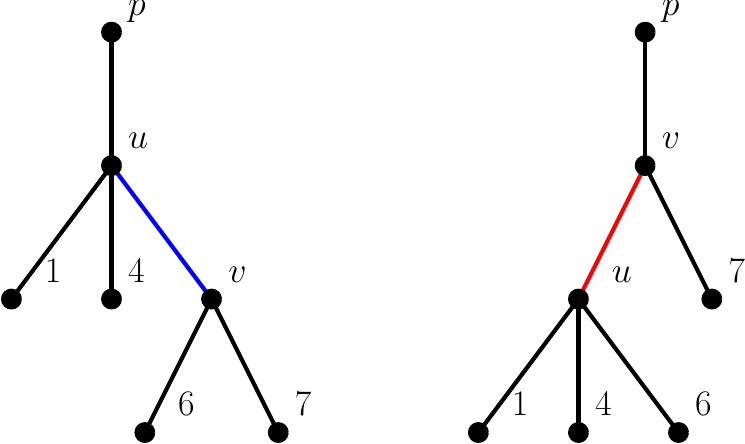}
\caption{\label{fig:rot}An example of rotation in the search tree of Figure~\ref{fig:tree}.}
\end{figure}

\begin{figure}
\subfigure[Before]{\includegraphics[page=12,scale=.5]{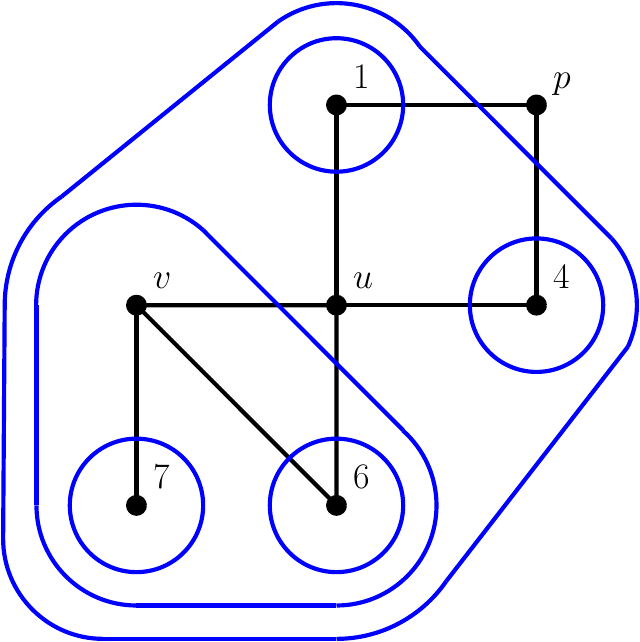}}
\hspace{1cm}
\subfigure[After]{\includegraphics[page=7,scale=.5]{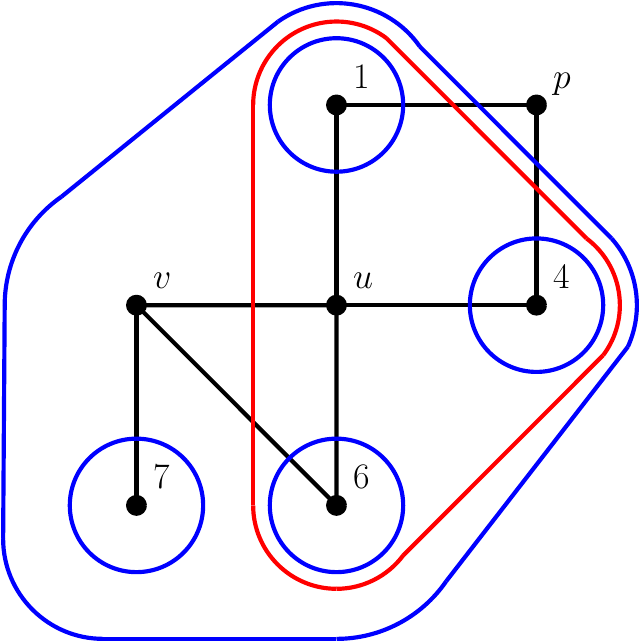}}
\caption{\label{fig:rottubes}A flip in a tubing corresponding to the rotation in Figure~\ref{fig:rot}.}
\end{figure}

Consider two vertices $u,v$ of a search tree $T$ in a connected graph $G$ such that $v$ is a child of $u$. Let $p$ be the parent of $u$ in $T$, and let $\bar{V}$ be the set of vertices of the subtree of $T$ rooted at $u$.
We can obtain another tree $T'$ from $T$ by a {\em rotation on $u$ and $v$} as follows:

\begin{itemize}
\item Make $u$ a child of $v$, and make $v$ a child of $p$.
\item Every remaining subtree $S$ of $u$ in $T$ remains a subtree of $u$ in $T'$.
\item For every subtree $S$ of $v$ in $T$, 
if $u$ is adjacent to a vertex of $S$ in $G$, 
then make $S$ a subtree of $u$ in $T'$; otherwise, $S$ remains a subtree of $v$ in $T'$.  
\end{itemize}

An example is given on Figure~\ref{fig:rot}.
Rotations in search trees are equivalent to {\em flips} in maximal tubings~\cite{MP15}.
Two maximal tubings are connected by a flip if and only if they differ by exactly two tubes, see Figure~\ref{fig:rottubes}.
The {\em rotation graph} $\mathcal{R} (G)$ of $G$ has the set of search trees on $G$ as vertex set, and is such that two search trees are
adjacent if and only if they differ by a single rotation. The rotation graph of $G$ is the skeleton of the graph associahedron of $G$.
The distance between two search trees in $\mathcal{R} (G)$ is referred to as the {\em rotation distance}. 
The diameter $\delta (\mathcal{R} (G))$ of $\mathcal{R} (G)$ is the largest rotation distance between two search trees on $G$.

\section{Diameter of tree associahedra}
\label{sec:statement}

Manneville and Pilaud~\cite{MP15} showed that $\delta (\mathcal{R} (G))$ is at least the number of edges of $G$.
They also mention the following upper bound, for which we give a short proof. 

\begin{lemma}
\label{lem:ub}
Let $G$ be a tree on $n$ vertices. Then $\delta (\mathcal{R} (G))=O(n\log n)$. 
\end{lemma}
\begin{proof}
A classical result due to Camille Jordan~\cite{J69} states that every tree has a vertex whose removal partitions the tree into connected
components of size at most $n/2$ each. Iteratively picking such a vertex as the root of the search tree yields the so-called 
{\em centroid decomposition} of the tree, which has height $O(\log n)$. Every tree can be turned into the search tree implementing
the centroid decomposition by first rotating the root up using at most $n-1$ rotations, then recursing on the subtrees. 
This yields an overall number of rotations in $O(n\log n)$. The centroid decomposition can in turn be transformed into any other tree
by applying $O(n\log n)$ rotations again.
\end{proof}

Manneville and Pilaud~\cite{MP15} posed the question of whether there exists a family of trees on $n$ vertices such that the diameter of their rotation graph is $\Omega (n\log n)$.
The question also appears in Ceballos et al.~\cite{CMPP15}.
We answer this question in the affirmative. We also confirm their conjecture that this lower bound is attained when $G$ is a complete binary tree. Together with Lemma~\ref{lem:ub}, it yields the following result.
\begin{theorem}
The maximum, over all trees $G$ on $n$ vertices, of $\delta(\mathcal{R}(G))$ is $\Theta (n\log n)$. 
\end{theorem}

\section{Proof}
\label{sec:proof}

\subsection{Preliminaries}

The proof of our lower bound uses induction on the number of vertices.
In order for this to be possible, we need to be able to {\em project} search trees, and rotation sequences on these search trees, on a subgraph of $G$. We show that there is a natural way to achieve this whenever $G$ is a tree and the subgraph is connected.

\begin{figure}
\includegraphics[page=8,scale=.5]{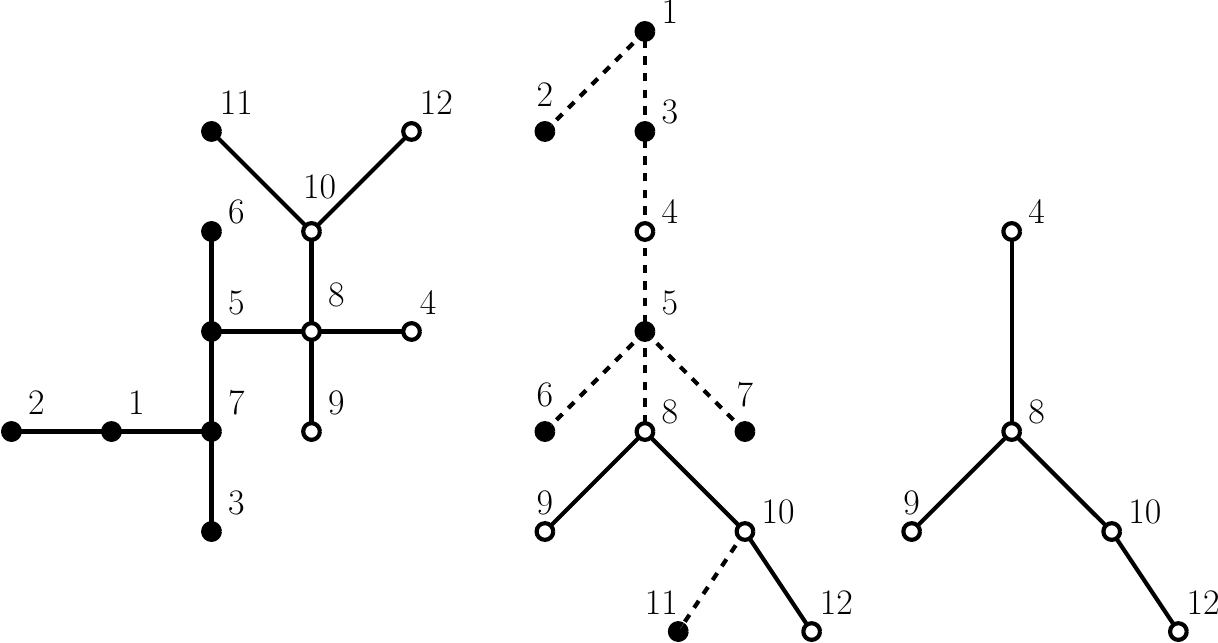}
\caption{\label{fig:proj}A tree $G$, a search tree $T$ on $G$, and its projection $T|S$ on $S=\{4,8,9,10,12\}$.} 
\end{figure}

As a preliminary step, let $x$ be a leaf (a degree-one vertex) of a tree $G=(V,E)$ with $|V|>1$, and consider a search tree $T$ on $G$.
Let us show how to construct a search tree on $G-x$ from $T$ that can be interpreted as the projection of $T$ on $V\setminus \{x\}$. 
In the search tree $T$, the vertex $x$ can be of three different types:
\begin{enumerate}
\item $x$ has a parent but no child,
\item $x$ has a parent and a single child,
\item $x$ is the root of $T$.
\end{enumerate}
In the first case, $x$ is also a leaf of $T$, and we can simply remove $x$ from $T$. In the second case, we can replace the two edges between $x$ and its parent and $x$ and its child by a single edge between the parent and the child. In the third case we remove $x$ from $T$ and choose its child as the new root. The tree that is obtained after this operation is a search tree on $G-x$. We refer to this operation as {\em pruning} $x$ from $T$.

A search tree on any connected subgraph $G[S]$ can be constructed by iteratively pruning a leaf of $G$ from $T$. We observe that given two leaves $x,y$ of $G$, we can prune them in any order and get the same search tree on $G-x-y$. The sequence in which the leaves are pruned is called a {\em vertex shelling order}. It is known that vertex shelling orders form the basic words of an antimatroid~\cite{KLS91,BZ92}, which are connected by adjacent transpositions. Hence the final tree that is obtained after a sequence of pruning only depends on the set of leaves that are pruned, and not on the specific order in which they are pruned.
This leads to the following inductive definition, illustrated in Figure~\ref{fig:proj}.

\begin{definition}[Projection of a search tree]
Consider a tree $G=(V,E)$ and a search tree $T$ on $G$.
Let $S\subseteq V$ be such that $G[S]$ is connected. 
The {\em projection} $T|S$ of $T$ on $S$ is the search tree on $G[S]$
obtained as follows:
\begin{itemize}
\item If $|V\setminus S|=1$ then $T|S$ is the result of pruning $x$ from $T$, where $x$ is the unique element of $V\setminus S$.
\item Otherwise, let $x$ be any vertex of $V\setminus S$ adjacent to a vertex of $S$ in $G$ and let $S'=S\cup\{x\}$. Then $T|S$ is the result of pruning $x$ from $T|S'$.
\end{itemize}
\end{definition}

We also observe the following property, which permits a more straightforward interpretation of the projection operation.

\begin{lemma}
Consider a tree $G=(V,E)$ and a search tree $T$ on $G$.
Let $S\subseteq V$ be such that $G[S]$ is connected. 
The projection $T|S$ of $T$ on $S$ is obtained by applying the following operation on each connected component $C$ of $T[V\setminus S]$:
\begin{itemize}
\item if two vertices of $T$ have a neighbor in $C$, then replace $C$ by a single edge,
\item otherwise, only one vertex $v$ of $T$ has a neighbor in $C$; delete $C$. If $C$ contains the root of $T$, choose vertex $v$ as the root of $T|S$.
\end{itemize}
\end{lemma}
\begin{proof}
Any connected component $C$ of $T[V\setminus S]$ can be constructed, starting from the projection $T|S$, by iteratively performing the inverse operations of those in the pruning, that is, either (1) adding a leaf, (2) splitting a single edge into two edges, or (3) adding a new root node with a single child. No sequence of such operations can increase the number of vertices of $V\setminus S$ that have a neighbor in $C$ above two. If the number of such vertices is equal to two, then the net effect of the pruning steps is to contract this component into a single edge. Otherwise, the component is simply removed.  
\end{proof}

We now consider rotation sequences, that is, sequences of rotation operations that can be applied starting with an initial tree $T$. 

\begin{lemma}[Projection Lemma]
Consider a tree $G=(V,E)$, two search trees $T$ and $T'$ on $G$, and a rotation sequence $\pi$ that transforms $T$ into $T'$.
Let $S\subseteq V$ be such that $G[S]$ is connected. 
Let the {\em projection} $\pi|S$ of $\pi$ on $S$ be obtained from $\pi$ by deleting every rotation on a pair $u,v$ such that at least one of $u,v$ lies in $V\setminus S$. 
Then $\pi|S$ is a rotation sequence that transforms $T|S$ into $T'|S$.
\end{lemma}
\begin{proof}
By induction, it is sufficient to prove the result for a single rotation, and only in the case where $S=V\setminus \{x\}$ for some leaf $x$ of $G$. 
Let $T'$ be the result of performing a rotation on $x$ in $T$.
Since $x$ is a leaf it must have at most one child in both $T$ and $T'$, and pruning it from one tree or the other yields the same search tree $T|S=T'|S$. 
Similarly, let $T'$ be the result of performing a rotation in $T$ that does not involve $x$. Then the same rotation transforms $T|S$ into $T'|S$.  
\end{proof}

Finally, our construction also relies on so-called {\em bit-reversal} permutations. 
These permutations play an important role in Fast Fourier Transform algorithms~\cite{K96}, but also in the analysis of performance of binary search trees~\cite{W89,DHIKP09}. We represent a permutation of $n$ elements by a sequence of integers in $\{0,1,\ldots ,n-1\}$, each appearing exactly once. Bit-reversal permutations get their name from the fact that they map each integer to the value obtained by reversing its binary representation.

\begin{definition}[Bit-reversal permutations]
The bit-reversal permutation of parameter one is $\sigma_1 = (0)$. The bit-reversal permutation $\sigma_k$ of parameter $k$ is defined by concatenating $2\sigma_{k-1}$ with $2\sigma_{k-1}+1$.
\end{definition}
In particular, we obtain:
\begin{eqnarray*}
\sigma_2 & = & (0,1)\\ 
\sigma_3 & = & (0,2,1,3)\\
\sigma_4 & = & (0,4,2,6,1,5,3,7)\\ 
\sigma_5 & = & (0,8,4,12,2,10,6,14,1,9,5,13,3,11,7,15).
\end{eqnarray*}

\subsection{Lower bound}

\begin{lemma}
\label{lem:main}
There exists a family of trees $\{ G_n \}$ on $n$ vertices such that $\delta(\mathcal {R}(G_n)) = \Omega (n\log n)$. 
\end{lemma}
\begin{proof}
Without loss of generality, we assume that $n=2^k-1$ for some integer $k\geq 1$. 
The family of trees $G_k=(V_k,E_k)$ is such that $G_1$ is a single vertex, and $G_k$ is composed of a single vertex connected to two subtrees isomorphic to $G_{k-1}$. Hence $G_k$ has the form of a complete binary tree. We label the $\ell \coloneqq (n+1)/2 = 2^{k-1}$ leaves of the tree $G_k$ by the integers from $0$ to $\ell - 1$ in their order in an inorder traversal of an arbitrary plane embedding of $G_k$.

\begin{figure}
\includegraphics[page=4,scale=.5]{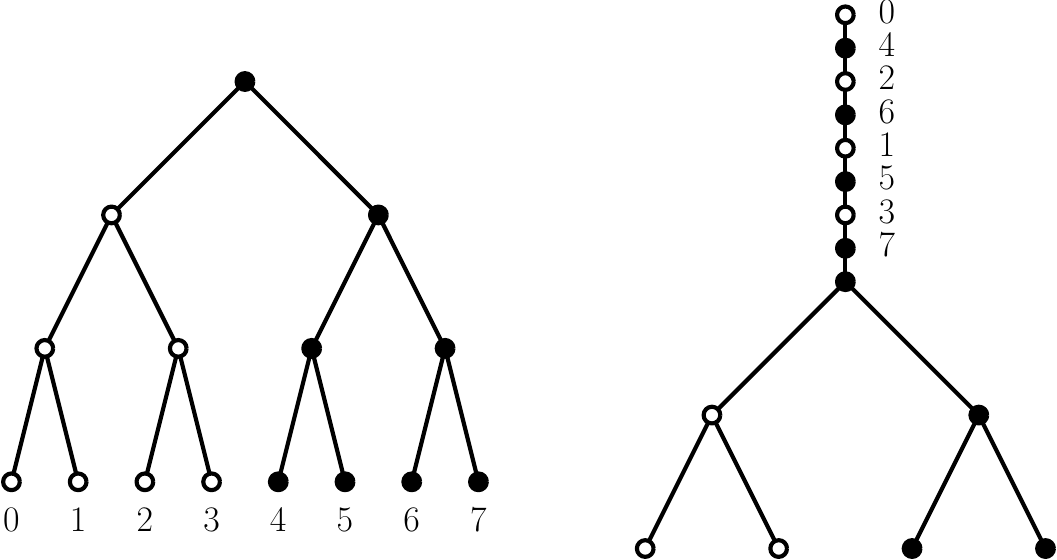}
\caption{\label{fig:example}The trees $T_4$ and $T'_4$. In $T'_4$, the leaves of $T_4$ are ordered by the permutation $\sigma_4$.}
\end{figure}

We now construct two search trees $T_k$ and $T'_k$ on $G_k$, such that the rotation distance between $T$ and $T'$ is $\Omega (n\log n)$.
The first tree $T_k$ is isomorphic to $G_k$.
More precisely, let $r$ be the unique vertex of degree two in $G_k$.
We first choose $r$ as the root of $T_k$. The remaining induced subgraph $G_k-r$ has two connected components isomorphic to $G_{k-1}$. Then we recursively choose the roots of the two subtrees in $T_k$ to be the vertices of degree two in each component, which are the two neighbors of $r$. 

The tree $T'_k$ has the leaf labeled 0 as root, followed by the sequence of leaves of $G_k$ in order of the bit-reversal permutation $\sigma_k$. Since those vertices are leaves of $G_k$, they all have exactly one child in $T'_k$. The last leaf is attached to the root of the remaining subgraph of $G_k$, which retains the same shape as in $T_k$. The trees $T_4$ and $T'_4$ are illustrated on Figure~\ref{fig:example}.

In what follows, we denote by $f(k)$ the rotation distance between $T_k$ and $T'_k$, with $f(1)=0$.
We consider a rotation sequence $\pi$ of length $f(k)$ that transforms $T_k$ into $T'_k$.
Let $A$ and $B$ denote the vertex sets of the two subtrees of the root $r$ of $T_k$.  
A rotation on $u,v$ in $\pi$ will be called an AA-rotation whenever both $u$ and $v$ belong to $A$.
BB-rotations are defined similarly. AB-rotations are rotations on a pair of vertices such that one is in $A$ and the other is in $B$.

\begin{figure}
\includegraphics[page=5,scale=.5]{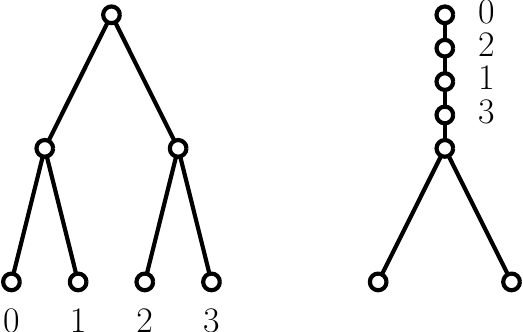}
\caption{\label{fig:split}The trees $T'_4|A\sim T_3$ and $T'_4|A \sim T'_3$.}
\end{figure}

\begin{claim}[Splitting]
The number of AA-rotations and the number of BB-rotations in $\pi$ are both at least $f(k-1)$.
\end{claim}
\begin{proof}
By definitions of the projection of a search tree and of the bit-reversal permutations $\sigma_k$, 
the pair of trees $(T_k|A, T'_k|A)$ is isomorphic to $(T_{k-1}, T'_{k-1})$, in the sense that the bijections between the vertex sets of the 
trees in each pair are the same. An example with $k=4$ is given on Figure~\ref{fig:split}.
The sequence $\pi | A$ contains exactly all the AA-rotations of $\pi$.
By applying the projection Lemma on $\pi$ with $S=A$, we know that the rotation sequence $\pi | A$ transforms $T_k|A$ into $T'_k|A$. 
Hence there is a rotation sequence of the same length transforming $T_{k-1}$ into $T'_{k-1}$. 
Therefore, by definition of $f$, the sequence $\pi | A$ must have length at least $f(k-1)$.
The same holds for the number of BB-rotations, by replacing $A$ by $B$.
\end{proof}

\begin{claim}[Merging]
The number of AB-rotations in $\pi$ is at least $2^{k-2}$.
\end{claim}
\begin{proof}
Let us consider an arbitrary search tree $T$ on $G_k$ and define its {\em alternation number} as the largest number of edges, on any path from the root to a leaf of $T$, with exactly one endpoint in $A$ and one endpoint in $B$. We make the following two observations.

(1) An AB-rotation in $T$ can increase its alternation number by at most two.

To see this, note that a rotation on $u$ and $v$ can only change the number of alternations of the paths involving $u$ or $v$, before or
after the rotation. The net effect of a rotation on a path consists of either swapping the two vertices, removing one, or inserting one.
Each of these changes modifies the alternation number by at most two.

(2) An AA or BB-rotation in $T$ cannot increase its alternation number.

Indeed, if the two vertices $u,v$ on which the rotation is made are both $A$ or $B$ vertices, the number of alternations after
removing one, inserting one, or swapping the two on a path leaves the alternation number unchanged.

We further observe that the alternation number of $T_k$ is 0, while the alternation number of $T'_k$, by construction of the bit-reversal permutation $\sigma_k$, is $\ell - 1$. Altogether, this proves that the number of AB-rotations in $\pi$ must be at least $\lceil (\ell - 1)/2\rceil = 2^{k-2}$. 
\end{proof}

Because the sets of AA, BB, and AB-rotations are disjoint, we obtain:
$$
f(k)\geq 2 f(k-1) + 2^{k-2} = \Omega (k2^k).
$$
Since $n=2^k-1$, this distance is $\Omega (n\log n)$, which concludes the proof.
\end{proof}

\bibliographystyle{plain}
\bibliography{references}

\end{document}